\documentclass{amsart}
\usepackage{amssymb}
\usepackage{amscd}
\usepackage{verbatim}
\usepackage{epsfig}

\begin{document}
\newcommand{\pa}{\partial}
\newcommand{\CI}{C^\infty}
\newcommand{\dCI}{\dot C^\infty}
\newcommand{\Hom}{\operatorname{Hom}}
\newcommand{\supp}{\operatorname{supp}}
\newcommand{\Op}{\operatorname{Op}}
\renewcommand{\Box}{\square}
\newcommand{\ep}{\epsilon}
\newcommand{\Ell}{\operatorname{Ell}}
\newcommand{\WF}{\operatorname{WF}}
\newcommand{\WFb}{\operatorname{WF}_{\bl}}
\newcommand{\WFsc}{\operatorname{WF}_{\scl}}
\newcommand{\diag}{\mathrm{diag}}
\newcommand{\sign}{\operatorname{sign}}
\newcommand{\Ker}{\operatorname{Ker}}
\newcommand{\Ran}{\operatorname{Ran}}
\newcommand{\Span}{\operatorname{Span}}
\newcommand{\sX}{\mathsf{X}}
\newcommand{\sH}{\mathsf{H}}
\newcommand{\sC}{\mathsf{C}}
\newcommand{\sk}{\mathsf{k}}
\newcommand{\codim}{\operatorname{codim}}
\newcommand{\Id}{\operatorname{Id}}
\newcommand{\id}{\operatorname{id}}
\newcommand{\cl}{{\mathrm{cl}}}
\newcommand{\piece}{{\mathrm{piece}}}
\newcommand{\bl}{{\mathrm b}}
\newcommand{\scl}{{\mathrm{sc}}}
\newcommand{\Psib}{\Psi_\bl}
\newcommand{\Psibc}{\Psi_{\mathrm{bc}}}
\newcommand{\Psibcc}{\Psi_{\mathrm{bcc}}}
\newcommand{\Psisc}{\Psi_\scl}
\newcommand{\Diff}{\mathrm{Diff}}
\newcommand{\Diffsc}{\Diff_\scl}
\newcommand{\PP}{\mathbb{P}}
\newcommand{\BB}{\mathbb{B}}
\newcommand{\RR}{\mathbb{R}}
\newcommand{\Cx}{\mathbb{C}}
\newcommand{\NN}{\mathbb{N}}
\newcommand{\R}{\mathbb{R}}
\newcommand{\sphere}{\mathbb{S}}
\newcommand{\codimY}{k}
\newcommand{\dimX}{n}
\newcommand{\cO}{\mathcal O}
\newcommand{\cS}{\mathcal S}
\newcommand{\cP}{\mathcal P}
\newcommand{\cF}{\mathcal F}
\newcommand{\cL}{\mathcal L}
\newcommand{\cH}{\mathcal H}
\newcommand{\cG}{\mathcal G}
\newcommand{\cU}{\mathcal U}
\newcommand{\cM}{\mathcal M}
\newcommand{\cT}{\mathcal T}
\newcommand{\cX}{\mathcal X}
\newcommand{\cY}{\mathcal Y}
\newcommand{\cR}{\mathcal R}
\newcommand{\loc}{{\mathrm{loc}}}
\newcommand{\comp}{{\mathrm{comp}}}
\newcommand{\Tb}{{}^{\bl}T}
\newcommand{\Sb}{{}^{\bl}S}
\newcommand{\Nb}{{}^{\bl}N}
\newcommand{\Tsc}{{}^{\scl}T}
\newcommand{\Ssc}{{}^{\scl}S}
\newcommand{\Vf}{\mathcal V}
\newcommand{\Vb}{{\mathcal V}_{\bl}}
\newcommand{\Vsc}{{\mathcal V}_{\scl}}
\newcommand{\Lambdasc}{{}^{\scl}\Lambda}
\newcommand{\etat}{\tilde\eta}
\newcommand{\scH}{{}^{\scl}H}
\newcommand{\Hsc}{H_{\scl}}
\newcommand{\Hscloc}{H_{\scl,\loc}}
\newcommand{\Hscd}{\dot H_{\scl}}
\newcommand{\Hscb}{\bar H_{\scl}}
\newcommand{\ff}{{\mathrm{ff}}}
\newcommand{\inter}{{\mathrm{int}}}
\newcommand{\Sym}{\mathrm{Sym}}
\newcommand{\be}[1]{\begin{equation}\label{#1}}
\newcommand{\ee}{\end{equation}}

\newcommand{\cutsphere}{\mathrm{PS}}

\newcommand{\level}{\mathsf{c}}
\newcommand{\foliation}{\mathsf{x}}
\newcommand{\loccoord}{y}
\newcommand{\Foliation}{\mathsf{X}}
\newcommand{\Loccoord}{\mathsf{Y}}

\newcommand{\bo}{\partial M} 
\newcommand{\zero}{^{(0)}}
\renewcommand{\r}[1]{(\ref{#1})} 
\newcommand{\mat}[4]{\left(\begin{array}{cc} #1 &#2\\#3 & #4 
\end{array}\right)}
\renewcommand{\d}{\mathrm{d}} 
\newcommand{\dsymm}{\mathrm{d}^{\mathrm{s}}}
\newcommand{\dsymmw}{\mathrm{d}^{\mathrm{s}}_\digamma}
\newcommand{\dsymmY}{\mathrm{d}^{\mathrm{s}}_Y}
\newcommand{\dsymmsc}{\mathrm{d}^{\mathrm{s}}_\scl}
\newcommand{\dsymmscw}{\mathrm{d}^{\mathrm{s}}_{\scl,\digamma}}

\setcounter{secnumdepth}{3}
\newtheorem{lemma}{Lemma}[section]
\newtheorem{prop}{Proposition}[section]
\newtheorem{proposition}{Proposition}[section]
\newtheorem{thm}{Theorem}[section]
\newtheorem{cor} {Corollary}[section]
\newtheorem{result}[lemma]{Result}
\newtheorem*{thm*}{Theorem}
\newtheorem*{prop*}{Proposition}
\newtheorem*{cor*}{Corollary}
\newtheorem*{conj*}{Conjecture}
\numberwithin{equation}{section}
\theoremstyle{remark}
\newtheorem{rem}{Remark}[section]
\newtheorem{remark} {Remark}[section]
\newtheorem*{rem*}{Remark}
\theoremstyle{definition}
\newtheorem{Def}{Definition}[section]
\newtheorem*{Def*}{Definition}

\newcommand{\mar}[1]{{\marginpar{\sffamily{\scriptsize #1}}}}
\newcommand\av[1]{}
\newcommand\gu[1]{}
\newcommand\mdh[1]{}

\title{Recovery of material parameters in transversely isotropic media}

\author[Maarten V. de Hoop, Gunther Uhlmann and Andras Vasy]{Maarten
  V. de Hoop, Gunther Uhlmann and Andr\'as Vasy}

\date{\today}

\address{Departments of Computational and Applied Mathematics, and Earth
  Sciences, Rice University, Houston, TX 77251-1892, U.S.A.}
\email{mdehoop@rice.edu}
\address{Department of Mathematics, University of Washington, 
Seattle, WA 98195-4350, U.S.A., University of Helsinki and Institute for Advanced Study,
HKUST, Clear Water Bay, Hong Kong, China}
\email{gunther@math.washington.edu}
\address{Department of Mathematics, Stanford University, Stanford, CA
94305-2125, U.S.A.}
\email{andras@math.stanford.edu}
\subjclass{35R30, 53C65, 35A17, 35S05, 53B40}

\begin{abstract}
In this paper we show that in anisotropic elasticity, in the
particular case of transversely isotropic media, under appropriate
convexity conditions, knowledge of the qSH
wave travel times determines the tilt of the axis of isotropy as well
as some of the elastic material parameters, and the knowledge of qP
and qSV travel times conditionally determines a subset of the
remaining parameters, in the sense if some of the remaining parameters
are known, the rest are determined, or if the remaining parameters
satisfy a suitable relation, they are all determined, under certain
non-degeneracy conditions. Furthermore, we give a precise description
of the additional issues, which are a subject of ongoing work, that
need to be resolved for a full treatment.
\end{abstract} 

\maketitle

\section{Introduction}

In this paper we show that in anisotropic elasticity, in the
particular case of transversely isotropic media, under appropriate
convexity conditions, knowledge of the qSH
wave travel times determines the tilt of the axis of isotropy as well
as some of the elastic material parameters, and the knowledge of qP
and qSV travel times conditionally determines a subset of the
remaining parameters, in the sense if some of the remaining parameters
are known, the rest are determined, or if the remaining parameters
satisfy a suitable relation, they are all determined, under certain
non-degeneracy conditions. Furthermore, we give a precise description
of the additional issues, which are a subject of ongoing work, that
need to be resolved for a full treatment.

The problem addressed in this paper has one of its primary application in
seismic tomography. In Earth's interior, the presence of anisotropy
has been widely recognized. In a classical (review) paper, Silver
described the seismic anisotropy beneath the continents
\cite{Silver-1996}. More recently, Romanowicz and Wenk
\cite{Romanowicz-Wenk-2017} described anisotropy in the deep
interior. The assumption of transverse isotropy with tilted symmetry
axis has played a dominant role in many studies ranging from Earth's
sedimentary basins, continental dynamics and subduction, deep mantle
flow and inner core.

The fundamental result of this paper is that the spatially varying
symmetry axis of a transversely isotropic elastic medium can be
locally recovered, under certain geometric conditions. However, in the
present analysis, the full recovery of elastic parameters requires some
interrelationships between them. Such relationships may be best
motivated by considering models that effectively generate these
parameters; these then provide possible physically, mechanically or
geologically based reductions of independent parameters. We briefly
mention a selection of examples of modeling procedures of this kind,
omitting references to a vast literature on the subject: (i)
differential effective medium theories, for which we refer to Norris
\cite{Norris-Callegari-Sheng-1985}; effective-medium-averaging
techniques to estimate the effective properties of a random sphere
pack while considering contact laws for adhesive contacts, rough
contacts, and smooth contacts, which were developed by Digby
\cite{Digby-1981} and Walton \cite{Walton-1987}, which later
culminated in the modeling of elastic properties of shales
\cite{Hornby-Schwartz-Hudson-1994} in sedimentary basins; (ii)
(sedimentary) layering-induced anisotropy in a simple calculus
formulated by Schoenberg and Muir \cite{Schoenberg-Muir-1989}, and
more general shape preferred orientation (SPO), considered by Garnero
and Moore \cite{Garnero-Moore-2004}, in its most basic form
originating from the study of a deformable elastic ellipsoid in a
far-field loaded matrix with different properties by Eshelby
\cite{Eshelby-1957, Eshelby-1959}; and (iii) strain-induced formation
of lattice preferred orientation (LPO). Indeed in Earth's upper mantle
it is generally accepted that seismic anisotropy results from LPO
produced by dislocation creep of olivine \cite{Zhang-Karato-1995},
while the mechanisms causing anisotropy in the inner core are still
under debate.

In order to state the results precisely, we work in an invariant
setting based on Riemannian geometry since this enables a cleaner and
conceptually clearer statement. Thus, there is a given background
metric $g_0$, which is typically the Euclidean metric; we denote the
dual metric and the dual metric function by $G_0$. In general,
anisotropic elasticity is described by a system whose principal
symbol, a tensor (matrix)-valued function on phase space, i.e., the
cotangent bundle, is non-scalar, i.e. is not a multiple of the
identity map. It can be diagonalized; the eigenvalues are the speed of
the various elastic waves. In isotropic elasticity, there are two
kinds of waves, P and S waves, with S waves corresponding (in spatial
dimension 3) to a multiplicity 2 (and P waves a simple) eigenvalue. In
anisotropic elasticity typically the S wave eigenspace is broken up,
at least in most parts of the cotangent bundle.  In transversely
isotropic elasticity there is a preferred axis, with respect to which
the principal symbol is rotationally symmetric relative to the
background metric $G_0$ lifted to the cotangent bundle. There are
three waves then, the qP waves, as well as the qSV and qSH waves, with
the latter corresponding to the `breaking up' of the S-waves. Of
these, the qSH waves behave much like in {\em isotropic} elasticity in
the sense that they are given the dual metric function of a Riemannian
metric, while the qP and qSV waves have a very different character.

One common parameterization of transversely isotropic materials, see \cite{Auld:Acoustic,Tsvankin:Seismic}, is via
the material constants $a_{11},a_{13},a_{33},a_{55}$ and $a_{66}$,
which are functions on the underlying manifold. In addition, there is
an axis of isotropy, which can be encoded by a vector field, or better
yet a one form $\omega$. The qSH `energy function' (dual metric
function) then depends on $a_{55},a_{66}>0$ and $\omega$.  Concretely,
using orthogonal coordinates relative to the metric $g_0$ (with $G_0$
the dual metric), and aligning the axis of isotropy with the third
coordinate axis, which is possible at any given point, the wave speed
of the qSH waves is given by a (squared!) Riemannian dual metric
$$
G=G_{qSH}=a_{66}(x)|\xi'|^2+a_{55}(x)\xi_3^2=a_{66}(x) G_0+(a_{55}(x)-a_{66}(x))\xi_3^2.
$$
This corresponds to a Riemannian metric
$$
g=g_{qSH}=a_{66}(x)^{-1}\, |dx'|^2+a_{55}(x)^{-1} \,dx_3^2=a_{66}(x)^{-1} g_0+(a_{55}(x)^{-1}-a_{66}(x)^{-1})\,dx_3^2,
$$
again at the point in question.
Thus, invariantly it has the form
$$
g=\alpha g_0+(\beta-\alpha)\omega\otimes\omega,
$$
i.e., the metric is a rank one perturbation of a conformal multiple of
the background (say, Euclidean) metric, with
$\alpha=a_{66}^{-1},\beta=a_{55}^{-1}$ functions on the base
manifold. Note that here $\beta-\alpha$ could be incorporated into
$\omega$ up to a sign; this formulation keeps the sign unspecified,
but then one should keep in mind that only the direction of $\omega$
matters. There is another reason to keep this form as explained below. Note
also that $g$ determines the span of $\omega$ if $\beta\neq\alpha$,
for the kernel of $\omega$ is well-defined (at any point in the
manifold) as the 2-dimensional subspace of the tangent space
restricted to which $g$ is a constant multiple of $g_0$.

Now, under appropriate assumptions, e.g.\ locally near the strictly
convex boundary, a Riemannian metric, $g$, can
be recovered from its boundary distance function up to
diffeomorphisms, as shown by Stefanov, Uhlmann and Vasy, recalled here in
Section~\ref{sec:sh}, see
\cite{Stefanov-Uhlmann-Vasy:Rigidity-Normal} for details. More
precisely, see Section~\ref{sec:sh}, the local determination indeed only uses the boundary distance
function, while the global result uses the lens relation, which also
keeps track of the direction of the geodesics at the two points on the
boundary at which they enter and exit the domain; in many cases these
are equivalent. Thus, if we know the qSH
wave travel times, then in fact we know $g$ above up to
diffeomorphisms (which are the identity at the boundary). A natural
question is whether this arbitrary diffeomorphism freedom is present
in our problem for the qSH wave travel times.

Formal dimension counting indicates that the space of Riemannian
metrics is 6 dimensional at each point, that of vector fields (or one
forms) is 3 dimensional at each point, so formally the space of
Riemannian metrics modulo diffeomorphisms is 3 dimensional. Now,
above, $\alpha,\beta$ are arbitrary functions, and $\omega$ is
arbitrary but only its direction matters, which means that the
parameter space at each point is 4 dimensional. This indicates that it
is unlikely that one can recover these four parameters from knowing
the corresponding Riemannian metric {\em up to diffeomorphisms}.

On the other hand, if one assumes that $\omega$ satisfies additional
conditions, this pointwise parameter space can be cut down, and the
problem may become formally determined. For instance, if $\omega$
always lies in the $dx_1$-$dx_3$ plane, this would be the case. (Note
that this includes the case when $\omega=dx_3$, in which case the
pointwise above form holds at least locally.)

An important property of a one-form, such as $\omega$, is its
integrability, or more precisely whether its kernel is an integrable
hyperplane distribution, which means that $\Ker\omega$ is the tangent
space of a smooth family of submanifolds, which are thus locally level
sets of a function $f$, so $\omega$ is a smooth multiple of $df$.  In
this case,
$$
   g=\alpha g_0+\gamma\,df\otimes\,df.
$$
In geological terms, this corresponds to a layered material with
layers given by the level sets of $f$. The integrability condition is
natural though not globally (that is, on planetary scale). LPO is one
mechanism that is consistent with this assumption while sedimentary
processes, compaction and deformation would yield the condition to
also hold true.

Our first theorem is:

\begin{thm}\label{thm:sh}
Consider the class of elastic problems in which $\Ker\omega=\Ker df$ is an integrable hyperplane
distribution on a manifold with boundary $M$, with $\omega$ not
conormal to $\pa M$ (so level sets of $f$ locally intersect $\pa M$ non-degenerately) and not orthogonal to $N^*\pa M$ relative to
$G_0$. Then, under the local, resp.\ global, convexity conditions for Riemannian
determination (up to diffeomorphisms), stated here in
Theorem~\ref{thm:local-impr} and Theorem~\ref{thm:global} of
Section~\ref{sec:sh}, $f,\alpha,\beta$ are locally, resp.\ globally, determined by the qSH
travel times, resp.\ qSH lens relations, and the labelling of the level sets of $f$ at the boundary.
\end{thm}

Thus, there is {\em no} diffeomorphism freedom in this problem,
unlike for the boundary rigidity problem in Riemannian geometry.

Since the qSH-wave speed does not depend on the remaining material
parameters, $a_{11},a_{13},a_{33}$, in order to go further we need to
consider qSV and qP waves.
Now, at a point with coordinates $g_0$-orthogonal at the point and such that the isotropy axis is
aligned with the $\tilde x_3$ axis  the Hamiltonians for the other waves take the form (with $\pm$ corresponding to
the choice of qP vs.\ qSV, and $G$ being twice what gives the actual
wave speeds)
\begin{equation}\begin{aligned}\label{eq:G-qP-qSV}
G_{qP/qSV}=(a_{11}+a_{55})|\tilde\xi'|^2&+(a_{33}+a_{55})\tilde\xi_3^2\\
&\pm\sqrt{\big((a_{11}-a_{55})|\tilde\xi'|^2+(a_{33}-a_{55})\tilde\xi_3^2\big)^2-4E^2|\tilde\xi'|^2\tilde\xi_3^2},
\end{aligned}\end{equation}
where
$$
E^2=(a_{11}-a_{55})(a_{33}-a_{55})-(a_{13}+a_{55})^2,
$$
see \cite{Schoenberg-deHoop:Approximate}.
(We will make the physically natural assumption that
$\max\{a_{55},a_{66}\}<\min\{a_{11},a_{33}\}$ throughout the paper.)
If one uses another coordinate system, $x_j$, often chosen orthogonal
at the point in question, and corresponding dual variables $\xi_j$,
the actual wave speed is given by the corresponding change of
variables. Thus, these wave speeds are no longer given by a quadratic
polynomial in the fibers of the cotangent bundle, and thus are not the
wave speeds of a Riemannian metric {\em unless $E=0$}, i.e., $E$
measures the departure from the Riemannian case (which is sometimes
called the `elliptic case' due to the quadratic polynomial
nature). (One can say that they are the wave speeds of a co-Finsler
metric due to the homogeneity with respect to dilations in the fibers
of the cotangent bundle, cf.\ \cite{Dahl:Finsler} for the terminology and for a detailed discussion.) Correspondingly, the Riemannian result,
\cite{Stefanov-Uhlmann-Vasy:Rigidity-Normal}, is not
applicable. Nevertheless, the analysis of that paper is based on the
study of a class of transforms which are microlocally weighted X-ray
transforms along curves, and even these general wave speeds fall in
this class, with the techniques introduced by Uhlmann and Vasy
\cite{Uhlmann-Vasy:X-ray} being applicable.

Following \cite{Uhlmann-Vasy:X-ray}, in this paper we work with a
function on $M$ with strictly convex level sets, and localize to
super-level sets of this function. We show that the modified and
localized `normal operators' that arise from the Stefanov-Uhlmann
pseudolinearization formula, which is valid for all Hamiltonian flows
and goes back to \cite{Stefanov-Uhlmann:Rigidity}, are scattering
pseudodifferential operators in Melrose's scattering
pseudodifferential algebra \cite{RBMSpec}, with the level set of the
function at which we stop playing the role of the boundary. (Thus,
this {\em artificial boundary} is the only one with analytic
significance, while the original boundary of $M$ simply constrains
supports.) In this algebra, whose properties are summarized in \cite[Section~3.2]{Stefanov-Uhlmann-Vasy:Rigidity-Normal}, there are two different (and somewhat
coupled) notions of ellipticity: that of the standard principal symbol
and that of the boundary principal symbol; the boundary principal
symbol at infinity in the fibers of the (scattering) cotangent bundle
is the same as the standard principal symbol at the boundary,
explaining the coupling. The standard principal symbol corresponds to
differentiable regularity, the boundary principal symbol to decay.

Now, there are three quantities we would still like to determine:
$a_{11},a_{33}$ and $E$, and we have two different wave speeds, the
qSV and the qP waves that we can use. While ideally one would like to
determine all of these at the same time, it is at this point natural
to formulate a theorem in which two of these three are regarded as
known, and the the third as unknown. Due to multiple points in the
cotangent space potentially corresponding to the same tangent vector
via the Hamilton map (a phenomenon that causes `wave triplication'),
we make an additional non-degeneracy hypothesis for the material, for
which we refer to Definition~\ref{def:non-deg-transverse}.

\begin{thm} \label{thm:p-sv}
Assume that the hypotheses of Theorem~\ref{thm:sh} hold, and that
$\nabla f$ is neither parallel, nor orthogonal to the artificial
boundary with respect to $g_0$. (This is automatic near $\pa M$ if the convex function is a
boundary defining function for $\pa M$, or a sufficiently small
perturbation of such.) Assume moreover that the transversely isotropic material is non-degenerate relative to a
convex foliation if qSV data are used below, with convexity of the
foliation always understood with respect to $G_{qP}$, resp.\
$G_{qSV}$, if qP, resp. qSV data are used below.

Then the modified and localized `normal operators'
arising from the Stefanov-Uhlmann formula are in Melrose's scattering
pseudodifferential operator algebra. Furthermore, the boundary
principal symbol is elliptic at {\em finite} points of the scattering
cotangent bundle for any one of $E^2,a_{11},a_{33}$ from the qP travel
data, and for $E^2$ (as well as $a_{11}$ and $a_{33}$ if $E^2>0$) from
the qSV travel data. Furthermore, for $a_{11}$ from the qP-travel time
data standard principal symbol ellipticity also holds.
\end{thm}

Note that the assumption that $df$ is not conormal to the artificial
boundary, i.e., $\nabla f$ is not orthogonal to it, means that the
span of $df$ has a non-degenerate image in the scattering cotangent
bundle; if $\rho$ defines the artificial boundary, this is that of the
scattering one-form $\rho^{-1} \, df$.

An immediate corollary, using the methods of \cite{Uhlmann-Vasy:X-ray,
  Stefanov-Uhlmann-Vasy:Boundary}, is:

\begin{cor}
Suppose that we are given the qSH-travel time data so that $\omega$,
$a_{55}$ and $a_{66}$ are determined already, and assume that the
hypotheses of Theorem~\ref{thm:p-sv} hold. Given $E^2$ and $a_{33}$,
the material parameter $a_{11}$ can be recovered from qP-travel time
data.
\end{cor}

Motivated by the discussion in the introduction on possible parameter
set reduction, by elimination we may invoke a functional relationship
where $a_{11}$ determines $a_{33}$ and $E^2$. This yields an
alternative to the corollary above:

\begin{cor} \label{thm:a11-fn}
Suppose that we are given the qSH-travel time data so that $\omega$,
$a_{55}$ and $a_{66}$ are determined already, and assume that the
hypotheses of Theorem~\ref{thm:p-sv} hold for both the qP and qSV
waves with the {\em same} convex foliation. Suppose also that we are
given $C^\infty$ functions $F,H:\RR\to\RR$ with $F'\geq 0$ such
that $a_{33} = F(a_{11})$ and $E^2 = H(a_{11})$. Then $a_{11}$ can be
recovered from the qP- and qSV-travel time data jointly.
\end{cor}

Finally, we show the precise nature of the obstruction to full
invertibility via elliptic analysis:

\begin{thm} \label{thm:double-chars}
For $a_{33},E^2$ from the qP or qSV travel data, as well as for $E^2$
and one of $a_{11}$ and $a_{33}$ jointly from the qP and qSV data, the
standard principal symbol is {\em not} elliptic, rather vanishes in a
non-degenerate quadratic manner along the span of $df$ at fiber
infinity in the scattering cotangent bundle.
\end{thm}

The explanation of the lack of ellipticity is very simple: In general,
for the normal operator's standard principal symbol computation at a
point $\zeta\in T^*_xM$, one takes a weighted average of certain
quantities evaluated at covectors for which the Hamilton vector field
for the relevant wave speed is annihilated by $\zeta$. Now, if
$\zeta=df$ is in the axis direction, the tangent vectors involved in
the integration correspond to covectors in the $g_0$-orthogonal plane,
i.e. with vanishing $\tilde\xi_3$ coordinate, and there the qP and
qSV wave speeds are insensitive to $a_{33}, E^2$ as these appear with
a prefactor $\tilde\xi_3^2$ above. The quadratic non-degeneracy also
corresponds to this: namely the relevant coefficient is a
non-degenerate multiple of $\tilde\xi_3^2$.

This means that the analytic framework for this inverse problem
involves double characteristics, which were studied in the paper of
Guillemin and Uhlmann \cite{Guillemin-Uhlmann:Oscillatory}. However,
here these need to be analyzed in the context of scattering
pseudodifferential operators, and the analysis must be {\em global} on
the manifold cut out by the artificial boundary.

Of course, we would like to determine all three of the remaining
parameters ideally. One may set up a system by adding a third row and
using different premultipliers, as was done in
\cite{Stefanov-Uhlmann-Vasy:Rigidity-Normal} to treat the normal
gauge, but one will certainly still have the double characteristic
phenomena at the minimum.

\medskip\medskip

\noindent
{\small The authors thank Joey Zou for pointing out an incorrect
  statement in Lemma~\ref{lemma:material-derivs} in an earlier version
  of the paper. The authors are very grateful to the Hong Kong University of
  Science and Technology (HKUST) and the University of Jyv\"askyl\"a
  for their hospitality during stays which led to this
  work. The authors also gratefully acknowledge partial support from the
  National Science Foundation, and G.U.\ is grateful for partial
  support from a Si Yuan
  professorship from the Institute for Advanced Study of HKUST, while
  M.V.d.H.\ and A.V.\ thank the Simons Foundation for partial support.}

\section{Proof of Theorem~\ref{thm:sh}}\label{sec:sh}
At the beginning of this section we recall the results of Stefanov,
Uhlmann and Vasy \cite{Stefanov-Uhlmann-Vasy:Rigidity-Normal}. The
simplest result to formulate is the local boundary rigidity result in
Riemannian geometry.

\begin{thm}(\cite[Theorem~1.1]{Stefanov-Uhlmann-Vasy:Rigidity-Normal})\label{thm:local-impr}
Suppose that $(M,g)$ is an $n$-dimensional Riemannian manifold with boundary, $n\geq
3$, and assume that $\bo$ is strictly convex (in the second
fundamental form
sense) with respect to each of the two metrics $g$ and $\hat g$ at some $p\in\bo$. If $d_g|_{U\times U}=d_{\hat g}|_{U\times U}$, for some neighborhood  $U$ of $p$  in $\pa M$, then  there is a neighborhood $O$ of $p$ in $M$  and a diffeomorphism $\psi:O\to \psi(O)$ fixing $\pa M\cap O$ pointwise  such
that $g|_O=\psi^*\hat g|_O$.

Furthermore, if the boundary is compact and everywhere strictly convex
with respect to each of the two metrics $g$ and $\hat g$ and
$d_g|_{\bo\times\bo}=d_{\hat g}|_{\bo\times\bo}$, then  there is a
neighborhood $O$ of $\bo$ in $M$ and a diffeomorphism $\psi:O\to \psi(O)$ fixing $\pa M\cap O$ pointwise  such
that $g|_O=\psi^*\hat g|_O$.
\end{thm}

The same paper also proves a global consequence of the local results,
assuming that $M$ is connected with non-trivial boundary. This global
statement requires a globally defined function $\foliation$ with level
sets which are strictly concave from the superlevel sets and which is
$\geq 0$ at a subset of $\pa M$; such
functions necessarily exist near the boundary but not necessarily
globally. (One can take for instance the negative of a boundary
defining function as a local example near the boundary, though this
does not localize within $\pa M$. See \cite{Uhlmann-Vasy:X-ray} for
more examples.) One also has global existence of such a function under appropriate
curvature conditions; see \cite{Stefanov-Uhlmann-Vasy:Rigidity-Normal}
for more details and references. (As an example, for domains in non-positively
curved simply connected manifolds, the distance to a point outside the
domain satisfies the concavity requirements.) This theorem uses the {\em lens
  relation}, which in addition to the distance between boundary points
keeps track of the directions at these points of geodesics connecting
them. For simple manifolds (strictly convex boundary and the geodesic
exponential map around each point is a diffeomorphism) the knowledge
of the boundary distance function $d_g|_{\bo\times\bo}$ and of the
lens relation is equivalent, see \cite{Michel}.

\begin{thm}(\cite[Theorem~1.3]{Stefanov-Uhlmann-Vasy:Rigidity-Normal})\label{thm:global}
Suppose that $(M,g)$ is a compact $n$-dimensional Riemannian manifold, $n\geq 3$, with strictly convex boundary, and $\foliation$ is a smooth function with non-vanishing
differential whose level sets are strictly concave from the superlevel sets, and $\{\foliation\geq 0\}\cap M\subset\pa M$.

Suppose also that $\hat g$ is a Riemannian metric on $M$ 
and suppose that
the lens relations of $g$ and $\hat g$ are the same.

Then there exists a diffeomorphism $\psi:M\to M$ fixing $\pa M$ such
that $g=\psi^*\hat g$.
\end{thm}

We start the proof of Theorem~\ref{thm:sh} by discussing some
  consequences of its integrability
hypothesis.

As we already mentioned, in general $\Ker\omega$, thus
in this case $\Ker df$ is well-defined, and so is its
$g$-orthocomplement, which is the same as the \newline $g_0$-orthocomplement,
since if a vector $W$ is $g_0$ orthogonal to an element $V$ of $\Ker
df$, then $df\otimes df(V,W)=df (V)\,df(W)=0$ shows the
$g$-orthogonality, and conversely. Moreover, taking $y_3=f$, one can introduce local
coordinates in which $\pa_{y_1},\pa_{y_2}$ are orthogonal to
$\pa_{y_3}$: one does this by defining $y_j$ on a level set of $f$,
and then extending them to be constant along integral curves of
$\nabla^g f$. Indeed, in this case $\pa_{y_1},\pa_{y_2}$ are tangent
to the level sets of $f$, for $\pa_{y_j}y_3=0$, $j=1,2$, while
$\pa_{y_3}$ is a multiple of $\nabla^g f$, which is orthogonal to the
$f$ level sets, hence to the $\pa_{y_j}$, $j=1,2$. Correspondingly, in
these coordinates, the metric takes the form
$$
g=\sum_{i,j=1}^2 a_{ij} \,dy_i\otimes dy_j+a_{33} \,dy_3\otimes dy_3.
$$
Notice that by the above remark, one has the same result if
$\nabla^{g_0} f$-integral curves are used, since they are also
orthogonal to the level sets of $f$, thus are simply
reparameterizations of the $\nabla^g f$-integral curves. Furthermore,
one can take {\em any} hypersurface transversal to $\nabla^{g_0} f$ to
define $y_1$ and $y_2$ originally. Thus, if $\nabla^{g_0} f$ is not
tangent to the boundary, i.e., $\omega$ is not $G_0$-orthogonal to $N^*\pa M$, {\em as we have assumed}, one can use the boundary for this
purpose.

Now suppose that two metrics $g$ and $\tilde g$ of this form are the same up to
a diffeomorphism $\Phi$ fixing the boundary, i.e.
$$
\tilde\alpha g_0+\tilde\gamma\,d\tilde f\otimes\,d\tilde
f=\Phi^*(\alpha g_0+\gamma\,df\otimes\,df).
$$
Since $\Ker df$ is determined by $g$, and $\Ker d\tilde f$ is
determined by $\tilde g$, $\Phi$ preserving the metrics implies that
the differential of $\Phi$ then will take $\Ker df$ and its
$g$, thus $g_0$-orthocomplement to $\Ker d\tilde f$ and its $\tilde
g$, thus $\tilde g_0$-orthocomplement. Using coordinates $y_j$ and
$\tilde y_j$ as above this means that $D\Phi$ is block-diagonal, with
the $(12)$ and $(3)$ blocks being non-trivial. This says that
$\pa_j\Phi_3=0$, $j=1,2$, and $\pa_3 \Phi_j=0$, $j=1,2$. Thus,
$\Phi_1$ and $\Phi_2$ are independent of $y_3$, while $\Phi_3$ is
independent of $y_1$ and $y_2$. So if one can make the argument that
$(\Phi_1,\Phi_2)$ is the identity at some point each $y_3$-curve then it
is so globally; moreover $\Phi_3$ simply relabels the level sets,
i.e., $\tilde y_3=\Phi_3(y_3)$. But
one can achieve this by choosing $y_1$ and $y_2$ on the boundary
(locally), {\em using that $\nabla^{g_0}f,\nabla^{g_0}\tilde f$ are
transversal to the boundary by our assumption}, and then choosing $\tilde y_1$ and
$\tilde y_2$ to be the same as $y_1,y_2$ there -- then at the
boundary, the (12) block of $D\Phi$ is the identity matrix. Moreover,
{\em as $\nabla f$ is not orthogonal to $\pa M$, i.e., $\omega$ is not
  conormal to $\pa M$, by our assumption}, the labelling of the
level sets of $f$ is determined by their value at the boundary (since
they intersect the boundary, and they do so non-degenerately), and
the same for $\tilde f$. Thus, in this case the diffeomorphism is the
identity, and thus $g$ is determined from the boundary distance
function (locally). Since $g$ in turn determines $\alpha,\gamma=\beta-\alpha$, this
proves Theorem~\ref{thm:sh}.

\section{Proof of Theorem~\ref{thm:p-sv} and Theorem~\ref{thm:double-chars}}

\subsection{The pseudolinearization formula and its basic properties}

To proceed, consider the Stefanov-Uhlmann pseudolinearization formula
which, as we recalled already, is valid for all
Hamiltonian flows and goes back to \cite{Stefanov-Uhlmann:Rigidity};
recall that in the isotropic setting one uses the momentum,
$\pa_\xi$, component of the Hamilton vector field to recover the
unknown wave speed. This in turn involves the position, $x$, derivative of the
effective Hamiltonian.
Concretely, see \cite[Section~7.2.2]{Stefanov-Uhlmann-Vasy:Rigidity-Normal}, the $\xi$-component of this formula for two Hamilton
vector fields $H_p$ and $H_{\tilde p}$ corresponding to two effective
Hamiltonians $p$ and $\tilde p$, denoting their flows by
$(X,\Xi)$ (with corresponding integral curve $\gamma$, exit time $\tau=\tau(x,\xi)$), resp.\ $(\tilde
X,\tilde\Xi)$, with $f=p-\tilde p$, takes the form
\begin{equation}\label{eq:SU-form}
J_i f(\gamma) \! = \! \int
(A^j_i(X(t),\Xi(t))\pa_{x^j}f(X(t),\Xi(t))+B_{ij}(X(t),\Xi(t))\pa_{\xi_j} f(X(t),\Xi(t)))\,dt=0
\end{equation}
with
\begin{equation*}\begin{aligned}
A^j_i(x,\xi)&=-\frac{\pa\tilde\Xi_i}{\pa\xi_j}(\tau(x,\xi),(x,\xi)),\\
B_{ij}(x,\xi)&=\frac{\pa\tilde\Xi_i}{\pa x^j}(\tau(x,\xi),(x,\xi)).
\end{aligned}\end{equation*}
Thus, at the boundary,
$$
A^j_i(x,\xi)=-\delta^j_i,\ B_{ij}(x,\xi)=0.
$$

Now suppose there is a function $P=P(x,\xi,\nu_1,\ldots,\nu_k)$
depending on parameters $\nu_j$, which are here the material
parameters $a_{ij}$, and corresponding to either wave speed $G_{qP/qSV}$, and suppose that
$$
p(x,\xi)=P(x,\xi,\nu_1,\ldots,\nu_k),\ \tilde
p(x,\xi)=P(x,\xi,\tilde\nu_1,\ldots,\tilde\nu_k),
$$
for two media with particular parameters $\nu_1,\ldots,\nu_k$, resp.\ $\tilde\nu_1,\ldots,\tilde\nu_k$.
Then
$$
p(x,\xi)-\tilde
p(x,\xi)=\sum_{j=1}^k(\nu_j(x)-\tilde\nu_j(x))E^j(x,\xi),
$$
with
$$
E^j(x,\xi)=\int_0^1\frac{\pa P}{\pa\nu_j}(s\nu_1(x)+(1-s)\tilde\nu_1(x),\ldots, s\nu_k(x)+(1-s)\tilde\nu_k(x),x,\xi)\,ds.
$$
Now, if these two media have the same
lens relations (and thus locally if they simply have the same travel times),
the Stefanov-Uhlmann identity gives
with $f_l(x)=\nu_l(x)-\tilde \nu_l(x)$, and now $f=(f_1,\ldots,f_k)$,
\begin{equation}\label{eq:SU-mod-form}
J_i f(\gamma)=\sum_{l=1}^k\int
(\hat A^{jl}_i(X(t),\Xi(t))\pa_{x^j}f_l(X(t))+\hat B^l_{i}(X(t),\Xi(t))f_l(X(t))\,dt=0
\end{equation}
with
\begin{equation*}\begin{aligned}
\hat A^{jl}_i(x,\xi)&=-\frac{\pa\tilde\Xi_i}{\pa\xi_j}(\tau(x,\xi),(x,\xi))E^l(x,\xi),\\
\hat B^l_{i}(x,\xi)&=-\frac{\pa\tilde\Xi_i}{\pa\xi_j}(\tau(x,\xi),(x,\xi)) \pa_{x^j}E^l(x,\xi)+\frac{\pa\tilde\Xi_i}{\pa x^j}(\tau(x,\xi),(x,\xi)) \pa_{\xi_j} E^l(x,\xi).
\end{aligned}\end{equation*}
Thus, at the boundary, where the $\nu_l$ and $\tilde\nu_l$ are equal,
\begin{equation*}\begin{aligned}
\hat A^{jl}_i(x,\xi)&=-\delta_i^j \frac{\pa P}{\pa\nu_l}(\nu_1(x),\ldots,\nu_k(x),x,\xi),\\
\hat B^l_{i}(x,\xi)&=-\delta_i^j \pa_{x^j}\frac{\pa P}{\pa\nu_l}(\nu_1(x),\ldots,\nu_k(x),x,\xi).
\end{aligned}\end{equation*}
Notice that now this is a {\em linear} transform in the $f_l$.

A convex foliation, by level sets of a function $\foliation$, often
plays a role in this paper. The level sets of a function $\foliation$ on $M$
are {\em concave}, or {\em concave from the superlevel sets}, or {\em
  convex from the sublevel sets} for a Hamiltonian $p$ if along
integral curves $\gamma(t)=(X(t),\Xi(t)))$ of $H_p$,
$\frac{d}{dt}(\foliation\circ X)(t)=0$ implies
$\frac{d^2}{dt^2}(\foliation\circ X)(t)>0$. This corresponds to the
assumptions in Section~3.2 of \cite{Uhlmann-Vasy:X-ray} above
Equation~(3.1). Convexity from the superlevel sets is defined
similarly, with the strict inequality reversed.

Instead of using the cotangent space for parameterizing the
bicharacteristics, one may want to use the tangent space instead. For
this one considers the Hamilton vector field map of the Hamiltonian
function $p$: the tangent vector
to a projected bicharacteristic $\gamma(t)=X(t)$ corresponding to the
bicharacteristic $(X(t),\Xi(t))$ is $\dot\gamma(t)=H_{X(t)}(\Xi(t))$,
where $H_x$ is the push-forward of the Hamilton vector field to the
base
$$
H_x(\xi)=\sum_j\frac{\pa p}{\pa\xi_j}(x,\xi)\pa_{x_j},
$$
where the notation indicates that for each base point $x$ we consider
it as a map
$$
\xi\mapsto H_x(\xi).
$$
When $p(x,.)$ is a quadratic polynomial (i.e. $p$ is a quadratic
polynomial in $\xi$), thus for Riemannian geometry and the qSH
transversely isotropic waves, this is a linear map, but in general it
is nonlinear. In order to parameterize the bicharacteristics, this
should be a map with a smooth inverse, at least locally along the
bicharacteristics we wish to use. This holds if $DH_x$ is
invertible. Explicitly, this differential is the Hessian matrix with
$ij$ entry $\frac{\pa^2 p}{\pa\xi_i\pa\xi_j}$. If $p(x,.)$ is a
positive definite quadratic polynomial, such as in Riemannian geometry
and qSH waves, then the Hessian matrix is positive definite, thus
invertible. Positive definiteness of the Hessian corresponds to strict
convexity of the level sets of $p$ from the side of the sublevel sets. In general, for interesting
examples of $p$ arising from qSV waves in transversely isotropic
materials, such as for the Greenhorn shale, see
e.g.\ \cite[Figure~2]{Schoenberg-deHoop:Approximate}, the strict
convexity may fail.

Since the general method of \cite{Uhlmann-Vasy:X-ray} uses curves that
are almost tangential to the level sets of the convex foliation, and
in many examples (supported by geodynamical considerations) the
tangent space of the level sets of the convex foliation may lie close
to the orthogonal plane to the isotropy axis, we start by remarking
that under easy to formulate conditions the qSV (and qP) level sets
are strictly convex there. (This is guaranteed if $E^2$ is not
exceedingly negative, while in Earth materials, typically, $E^2 \ge 0$
\cite{Thomsen}.) In the following lemma, the tilded coordinates
correspond to the transverse isotropy with the third coordinate
corresponding to its axis, as in the introduction; note that the
Hamilton vector field being invariantly defined, it makes no
difference in what coordinates we consider the map $H_x$. (We also
recall the standing assumption $\max\{a_{55},a_{66}\} <
\min\{a_{11},a_{33}\}$ here.)

\begin{lemma}\label{lemma:non-deg-diff-xi3}
Suppose that either $p=p_+=G_{qP}$, or instead $p=p_-=G_{qSV}$ and
$$
a_{33} (a_{11}-a_{55})>(a_{13}+a_{55})^2.
$$
Then the map $\tilde \xi\mapsto H_{\tilde x}(\tilde\xi)=\sum_j\frac{\pa p}{\pa\tilde\xi_j}\pa_{\tilde
  x_j}$ has an invertible differential at $\tilde\xi_3=0$, and indeed
the level sets of $p$ are strictly convex (from the sublevel sets) nearby.
\end{lemma}

\begin{rem}
This lemma also plays an important role below in studying the precise
degeneracy in determining various material parameters from various waves.
\end{rem}

\begin{rem}
Notice that if $E^2\geq 0$, the right-hand side is
$\leq(a_{11}-a_{55})(a_{33}-a_{55})$, so the inequality in the
statement of the lemma is automatically true.
\end{rem}

\begin{proof}
We just need to compute the Hessian matrix $\frac{1}{2}\frac{\pa^2
  p_\pm}{\pa\tilde\xi_i\pa\tilde\xi_j}$ and show that it is positive definite
when $\tilde\xi_3=0$. But this Hessian is diagonal, with a
multiplicity 2 entry for the first 2 components. At $\tilde\xi_3=0$ the
multiplicity two entry is particularly easy to evaluate as one may
simply set $\tilde\xi_3=0$ prior to differentiation to obtain
$$
(a_{11}+a_{55})\pm (a_{11}-a_{55}),
$$
which are positive.
So it remains to evaluate the multiplicity one entry, namely $\frac{1}{2}\frac{\pa^2
  p_\pm}{\pa\tilde\xi_3^2}$. Again, this simplifies as after the first
differentiation we may set all terms with a $\tilde\xi_3^2$ factor to
$0$, i.e., we just need to differentiate
$$
(a_{33}+a_{55})\tilde\xi_3\pm\frac{(a_{33}-a_{55})(a_{11}-a_{55})|\tilde\xi'|^2-2E^2|\tilde\xi'|^2}{(a_{11}-a_{55})|\tilde\xi'|^2}\tilde\xi_3,
$$
which is
$$
(a_{33}+a_{55})\pm\frac{(a_{33}-a_{55})(a_{11}-a_{55})-2E^2}{a_{11}-a_{55}}.
$$
As $E^2=(a_{11}-a_{55})(a_{33}-a_{55})-(a_{13}+a_{55})^2$, this
simplifies to
$$
(a_{33}+a_{55})\pm\frac{2(a_{13}+a_{55})^2-(a_{11}-a_{55})(a_{33}-a_{55})}{a_{11}-a_{55}},
$$
which is
$$
2a_{55}+\frac{2(a_{13}+a_{55})^2}{a_{11}-a_{55}},
$$
thus always positive, for the $+$ sign,
and is
$$
2a_{33}-\frac{2(a_{13}+a_{55})^2}{a_{11}-a_{55}},
$$
for the $-$
sign, which is positive if
$$
a_{33} (a_{11}-a_{55})>(a_{13}+a_{55})^2.
$$
This completes the proof.
\end{proof}

In general, as already mentioned, we do not have strict convexity of
the level sets, and as the Hessian changes signature, $H_x$ ceases to
have an invertible differential along some submanifolds. Globally,
this results in $H_x$ not being injective, giving rise to phenomena
such as triplication (higher multiplicities cannot occur in the case
of transverse isotropy) where a given (normalized) tangent vector is
the image of multiple covectors. However, {\em for qP waves strict
  convexity (from sublevel sets) always holds} \cite[p.168]{Chapman}, \cite{Payton}, and
in general this phenomenon motivates the following definition:

\begin{Def}\label{def:non-deg-transverse}
A transversely isotropic material is non-degenerate relative to a
convex foliation (concave from the superlevel sets for $G_{qSV}$) if for each point $x$ and each vector $v$ tangent to
the convex foliation at the point $x$ there is a covector $\xi$ in the
cotangent space over $x$ such that $H_{x}(\xi)=v$ and the map $H_{x}$
has invertible differential at $\xi$, with $H_x$ arising from
$G_{qSV}$. A transversely isotropic material is non-degenerate
provided the statement above holds for all $v$ (and not just $v$
tangent to a particular convex foliation).
\end{Def}

Lemma~\ref{lemma:non-deg-diff-xi3}, under the assumed condition, thus
shows that if the transverse isotropy orthogonal planes are close to
the tangent spaces to a convex foliation, then the material is
non-degenerate relative to the convex foliation.

In a non-degenerate, relative to a convex foliation, material, one may
always consider, at least locally, the bicharacteristics to be
parameterized by tangent vectors. This is useful both in order to
localize to almost tangent to the convex foliation vectors and also to
analyze the transform: stationary phase computations, discussed below,
use the natural pairing between covectors at which principal symbols
are evaluated and tangent vectors to the projected bicharacteristics
being used. This approach also has the advantage of connecting better
to the notation of \cite{Uhlmann-Vasy:X-ray,
  Stefanov-Uhlmann-Vasy:Boundary, Stefanov-Uhlmann-Vasy:Rigidity-Normal}.

Thus, from now on, {\em we assume that the material is non-degenerate
relative to the fixed convex foliation.} We then define a transform
$\tilde L$ from the cotangent space, which is a transform of the
form
$$
\tilde L=\sum_i \Psi_i L_i \Psi_i^{-1}\tilde\phi_i,
$$
$\tilde\phi_i$ a cutoff supported in a
region on a neighborhood of which $H_x$ is smoothly invertible, and $L_i$ is a
transform, discussed below, from the tangent space, where the local
identification $\Psi_i$ is given by pull-back by the Hamilton vector field map
$H_x$, and $\Psi_i^{-1}$ is the pull-back by the local inverse
$H_x^{-1}$. Concretely, we cover a neighborhood of the tangent space
of the convex foliation with open sets $O_i$ on each of which $H_x^{-1}$
exists as a smooth map with image $\tilde O_i$ in the cotangent space, and take a corresponding partition of unity $\phi_i$ (so
$\sum_i\phi_i=1$ on a smaller neighborhood of the tangent space of the
convex foliation), and let $\tilde\phi_i$ be defined as $H_x^*\phi_i$
on $\tilde O_i$ (with support in a compact subset of $\tilde O_i$) and
$0$ outside.

{\em In order to avoid overburdening the notation, since all arguments
below are local we drop the index $i$, and simply write $L$, and
understand that $H_x^{-1}$ refers to the localized inverse for
$H_x:\tilde O_i\to O_i$.}

Following \cite{Uhlmann-Vasy:X-ray},
we use coordinates $(x_1,x_2,x_3)=(y,x)=z$ in which $x=x_3$ are the level sets
are the convex foliation, and $x_3=0$ is the artificial boundary,
We write tangent vectors as $\lambda\,\pa_x+\omega\,\pa_y$, and the
projected bicharacteristic corresponding to such a tangent vector at
$z$ as $\gamma_{z,\lambda,\omega}=\gamma_{x,y,\lambda,\omega}$.
One then considers an operator of the form $LJ$, where $L$ is a
slightly modified version of $J^*$, and where $L$ cuts off at the
artificial boundary, see \cite{Uhlmann-Vasy:X-ray}:
$$
(Lv)(z)=x^{-2}\int\chi(\lambda/x)v(\gamma_{z,\lambda,\omega})\, d\lambda\,d\omega,
$$
cf.\ \cite{Stefanov-Uhlmann-Vasy:Rigidity-Normal}, the displayed
equation below (3.1) (this differs from \cite{Uhlmann-Vasy:X-ray} in
normalization). Here $\chi$ is a non-negative smooth compactly
supported function, $\chi(0)>0$, which is appropriately chosen as in
\cite{Uhlmann-Vasy:X-ray}, see also Lemma~\ref{lemma:finite}. The particular smooth measure $d\lambda\,d\omega$
is irrelevant; any other positive definite smooth measure will
do. Note that the measure has nothing to do with the Euclidean metric
$g_0$ (which plays a role in the transverse isotropy!)
in particular, and similarly the coordinates $x_j$ have nothing to do
with Euclidean metric.

The main terms in \eqref{eq:SU-mod-form} are the $\hat A^{jl}_i(x,\xi)$ terms; the others
can be absorbed into these by Poincar\'e inequalities, at least {\em if the
$\hat A^{jl}_i(x,\xi)$ terms are non-degenerate}, see \cite{Stefanov-Uhlmann-Vasy:Boundary}. To leading
order at the boundary these decouple due to the $\delta_i^j$, so one
is essentially working on a microlocally weighted X-ray transform
combining the differences of the unknown material parameters; more
precisely one has a transform for each derivative of the combinations
of the differences
of these unknown material parameters. (One of course has to deal with
these transforms together as done in
\cite{Stefanov-Uhlmann-Vasy:Boundary} and follow-up papers.)
Thus, one may consider the simplified transforms
\begin{equation}\label{eq:SU-simple-form}
\tilde J \tilde f(\gamma)=\sum_{l=1}^k\int
(\tilde A^{l}(X(t),\Xi(t))\tilde f_l(X(t))\,dt=0
\end{equation}
with
\begin{equation*}\begin{aligned}
\tilde
A^{l}(x,\xi)&=-\frac{\pa\tilde\Xi_j}{\pa\xi_j}(\tau(x,\xi),(x,\xi))E^l(x,\xi),\\
\tilde A^{l}(x,\xi)&=- \frac{\pa
  P}{\pa\nu_l}(\nu_1(x),\ldots,\nu_k(x),x,\xi)\ \text{at}\ \pa M,\\
\tilde f_l&=\pa_j f_l,
\end{aligned}\end{equation*}
with $j$ fixed.

The following proposition is proved completely analogously to \cite[Corollary~3.1]{Stefanov-Uhlmann-Vasy:Boundary}:

\begin{prop}\label{prop:J-tilde-J}
If the operator $e^{-\digamma/x}L\tilde Je^{\digamma/x}$ is elliptic
when considered as a map from a single component of $\tilde f_l$
(i.e. with the others set to $0$), the ellipticity of the full
operator $e^{-\digamma/x}LJe^{\digamma/x}$ follows as a map for the corresponding component,
provided the artificial boundary is sufficiently close to $\pa
M$.
\end{prop}

Roughly
speaking, the hypothesis of this proposition says that {\em ignoring coupling} one can recover
the derivatives of $f_l$ (due to ellipticity, choosing the artificial
boundary sufficiently close to $\pa M$),
which then, as the conclusion states, allows one to recover $f_l$, though due to the coupling in $LJ$, a
Poincar\'e inequality based argument is needed as in
\cite[Corollary~3.1]{Stefanov-Uhlmann-Vasy:Boundary}. Because of this
proposition, in what follows we concentrate on properties of $L\tilde J$.

Now, with $p_\pm=G_{qP/qsV}$ (with $+$ for qP; notice that $p_\pm$
stands for $P$ above), and with tilded
coordinates corresponding to the transversely isotropic structure,
{\em not} the convex foliation, we have from \eqref{eq:G-qP-qSV} that
\begin{equation}\begin{aligned}\label{eq:material-derivs}
\frac{\pa p_\pm}{\pa E^2}&=\mp\frac{2|\tilde\xi'|^2\tilde\xi_3^2}{\sqrt{\big((a_{11}-a_{55})|\tilde\xi'|^2+(a_{33}-a_{55})\tilde\xi_3^2\big)^2-4E^2|\tilde\xi'|^2\tilde\xi_3^2}},\\
\frac{\pa p_\pm}{\pa a_{11}}&=|\tilde\xi'|^2\Big(1\pm\frac{(a_{11}-a_{55})|\tilde\xi'|^2+(a_{33}-a_{55})\tilde\xi_3^2}{\sqrt{\big((a_{11}-a_{55})|\tilde\xi'|^2+(a_{33}-a_{55})\tilde\xi_3^2\big)^2-4E^2|\tilde\xi'|^2\tilde\xi_3^2}}\Big),\\
\frac{\pa p_\pm}{\pa a_{33}}&=\tilde\xi_3^2\Big(1\pm\frac{(a_{11}-a_{55})|\tilde\xi'|^2+(a_{33}-a_{55})\tilde\xi_3^2}{\sqrt{\big((a_{11}-a_{55})|\tilde\xi'|^2+(a_{33}-a_{55})\tilde\xi_3^2\big)^2-4E^2|\tilde\xi'|^2\tilde\xi_3^2}}\Big).
\end{aligned}\end{equation}

We summarize some immediate definiteness properties (keep in mind the background assumption that
$\max\{a_{55},a_{66}\}<\min\{a_{11},a_{33}\}$) of these material derivatives in
a lemma:

\begin{lemma}\label{lemma:material-derivs}
We have:
\begin{enumerate}
\item
$\frac{\pa p_+}{\pa a_{11}}$ is a positive definite multiple of
$|\tilde\xi'|^2$, thus is in particular non-negative.
\item
$\frac{\pa p_+}{\pa a_{33}}$ is a positive definite multiple of
$\tilde\xi_3^2$, thus is in particular non-negative.
\item
$\frac{\pa p_\pm}{\pa E^2}$ are positive definite multiples of
$\mp|\tilde\xi'|^2\tilde\xi_3^2$, thus is in particular non-positive/non-negative.
\item
If $E^2>0$, $\frac{\pa p_-}{\pa a_{11}}$ is negative definite
multiple of $|\tilde\xi'|^2$ away from $\tilde\xi_3=0$ and from
$\tilde\xi'=0$, and is everywhere non-positive; the analogous
statement holds if $E^2<0$ with `positive' and `negative' reversed.
\item
If $E^2>0$, $\frac{\pa p_-}{\pa a_{33}}$ is a negative definite
multiple of $\tilde\xi_3^2$ away from $\tilde\xi_3=0$ and from $\tilde\xi'=0$, and is everywhere non-positive; the analogous
statement holds if $E^2<0$ with `positive' and `negative' reversed.
\end{enumerate}
\end{lemma}

\begin{rem}
Notice that when $E^2=0$, $\frac{\pa p_-}{\pa a_{11}}=0$ and
$\frac{\pa p_-}{\pa a_{33}}=0$, i.e., $a_{11}$ and $a_{33}$ affect
only $p_+$. (In isotropic elasticity, $a_{11}=a_{33}=\lambda+2\mu$ in
terms of the Lam\'e parameters, while $a_{55}=\mu$, and $E^2=0$,
$a_{13}=\lambda$, so this is the statement that the S waves are
insensitive to $\lambda$.)
\end{rem}

As we recall below, the derivatives in \eqref{eq:material-derivs} will
be evaluated at the points in the cotangent bundle on the support of
the cutoff $\chi$ in $L$, which means that points near the image of
the tangent space of the level sets of the convex foliation under the
local inverses $H_x^{-1}$.  Furthermore, for the principal symbol
computation for $L\tilde J$ at a point $(x,\zeta)$, by stationary
phase, one actually needs the base tangent vector to the
bicharacteristic, $H_x(\xi)=\lambda\,\pa_x+\omega\cdot\pa_y$, be
annihilated by $\sum_j\zeta_j\,dx_j$, though this needs to be suitably
interpreted at the artificial boundary since $\zeta$ is actually a
scattering covector.

\subsection{Principal symbols}

Concretely, for the standard principal symbol computation one writes the
projected bicharacteristics through $x,y$, with tangent vector
$\lambda\pa_x+\omega\cdot\pa_y$ at that point, in the form
$$
\gamma(t)=\gamma_{x,y,\lambda,\omega}(t)=(\gamma^{(1)}_{x,y,\lambda,\omega}(t), \gamma^{(2)}_{x,y,\lambda,\omega}(t))=
(x+\lambda t+\alpha
t^2+O(t^3),y+\omega t+O(t^2)),
$$
where the $O$'s are understood to mean the indicated prefactor times a
smooth function of all variables,
$$
(x_1,x_2,x_3)=((x_1,x_2),x_3)=(y,x)=z,
$$
where
$\alpha=\alpha(x,y,\lambda,\omega)$, see Section~4 of
\cite{Stefanov-Uhlmann-Vasy:Rigidity-Normal}, including expanding all
the $O$ error terms into smooth functions. Thus, to match the
notation of this paper and \cite{Uhlmann-Vasy:X-ray}, the scalar
function $x$ stands for $x_3$ in terms of the vector coordinates
$(x_1,x_2,x_3)$, and yet it is written as the first component of $\gamma$. (Technically, that
section of
\cite{Stefanov-Uhlmann-Vasy:Rigidity-Normal} is in the more complicated one form setting, so there are
various slight simplifications in the computations for the present
purposes.)
Then
\begin{equation*}\begin{aligned}
&(\gamma(t),\dot\gamma(t))=(x+\lambda t+\alpha
t^2+O(t^3),y+\omega t+O(t^2),\lambda+2\alpha
t+O(t^2),\omega+O(t)),
\end{aligned}\end{equation*}
and scaling
$\hat\lambda=\lambda/x$, $\hat t=t/x$, as relevant below for the
oscillatory integral, gives
\begin{multline*}
(\gamma(x\hat t),\dot\gamma(x\hat t))
=(x+x^2(\hat\lambda \hat t+\alpha
\hat t^2+xO(\hat t^3)),
\\
y+x(\omega \hat t+xO(\hat
t^2)),x(\hat\lambda+2\alpha
\hat t+xO(\hat t^2)),\omega+xO(\hat
t)).
\end{multline*}
The weight $\hat A^{jl}_i$ is on the phase space due to the
Hamiltonian dynamics used in the Stefanov-Uhlmann formula, so the
tangent vector $\dot\gamma(t)$ needs to be converted to a covector via
the local inverse $H_{\gamma(t)}^{-1}$ of the Hamilton vector field map.

The symbol whose left quantization is
\begin{equation}\label{eq:def-of-conjugated-LJ}
A_{l,\digamma}=e^{-\digamma/x}L\tilde Je^{\digamma/x},
\end{equation}
considered restricted to
functions with only non-vanishing $l$th components,
is, cf.\ \cite[Equation~(4.9)]{Stefanov-Uhlmann-Vasy:Rigidity-Normal},
combined with the weights discussed in \cite[Equation~(6.15)]{Stefanov-Uhlmann-Vasy:Rigidity-Normal},
\begin{multline} \label{eq:1-form-kernel-sc-form}
a_{l,\digamma}(x,y,\zeta)=
\int e^{-\digamma/x}
e^{\digamma/\gamma^{(1)}_{x,y,\lambda,\omega}(t)}
e^{i(\zeta_3/x^2,\zeta'/x)\cdot(\gamma^{(1)}_{x,y,\lambda,\omega}(t)-x,
  \gamma^{(2)}_{x,y,\lambda,\omega}(t)-y)}
\\
\tilde A^{l}(\gamma(t),H_{\gamma(t)}^{-1}(\dot\gamma(t)))\chi(\lambda/x) 
\,dt\,d\lambda\,d\omega,
\end{multline}
where $\zeta$ is the scattering
coordinate (so covectors are
$\zeta_3\frac{dx}{x^2}+\zeta'\cdot\frac{dy}{x}$). After rescaling
$\lambda$ and $t$, this becomes a non-degenerate oscillatory integral with critical points at the codimension 2 submanifold
\begin{equation}\label{eq:phase-crit-set}
\hat t=0,\ \zeta_3\hat\lambda+\zeta'\cdot\omega=0.
\end{equation}
Note that if $\zeta'$ is large relative to $\zeta_3$, i.e., if we stay
away from a conic neighborhood of $\zeta'=0$, one can use $\hat t$ and
$\omega^\parallel$ as the variables in which the stationary phase is
performed, where $\omega$ is decomposed relative to $\zeta'$ into a
parallel and a perpendicular vector; then $\hat\lambda,\omega^\perp$
parameterize the critical set. On the other hand, if $\zeta_3$ is
large relative to $\zeta'$, then one can use $\hat t$ and
$\hat\lambda$ as the variables in which stationary phase is performed;
then $\omega$ parameterizes the critical set.
Hence, substituting the above expressions for $\gamma,\dot\gamma$, we
conclude that up to errors that are
$O(x\langle\zeta\rangle^{-1})$ relative to
the a priori order, $(-1,0)$, arising from the $0$th order symbol in
the oscillatory integral and the 2-dimensional space in which the
stationary phase lemma is applied,
\begin{multline} \label{eq:aj-1-form-pre-xi}
a_{l,\digamma}(x,y,\zeta)
\\
=\int e^{i(\zeta_3(\hat\lambda\hat t+\alpha\hat t^2)+\zeta'\cdot(\omega\hat
t))}e^{-\digamma(\hat\lambda\hat t+\alpha\hat
t^2)}\tilde A^{l}(x,y,H_{x,y}^{-1}(x\hat\lambda,\omega))\chi(\hat\lambda)\,d\hat t\,d\hat \lambda\,d\omega.
\end{multline}
At this point we apply the stationary phase lemma. Up to an overall
elliptic factor, this results in an integral of
$$
\tilde A^{l}(x,y,H_{x,y}^{-1}(x\hat\lambda,\omega))\chi(\hat\lambda)
$$
over the critical set with a positive weight. In particular, if this
has a fixed indefinite sign, e.g.\ is non-negative at all points of, and
is actually definite (positive in the example) at one point of, the
critical set, the resulting operator is elliptic.
Note that we are using $\chi\geq 0$ with $\chi(0)>0$, so if $\tilde
A^{l}$ has a fixed indefinite sign, the key
question is if there is a point on the critical set with $\hat\lambda$
small at which $\tilde
A^{l}$ has a definite sign.

Notice \eqref{eq:phase-crit-set} states that if $\zeta_3=0$, the
tangent vector $\omega$ is annihilated by $\zeta'$ and $\hat\lambda$ is
arbitrary (with the localizing cutoff $\chi$ keeping it in a compact region);
otherwise $\omega$ is actually arbitrary, though if $\zeta_3$ is small
relative to $\zeta'$,
this requires a very large $\hat\lambda$, which may fall outside the
support of the cutoff $\chi$; in any case as long as the cutoff is
non-trivial at zero, a neighborhood of those $\omega$ annihilated by
$\zeta'$ is relevant. Notice that at $x=0$, {\em regardless of the value of
  $\hat\lambda$}, the corresponding tangent vector is just
$(x\hat\lambda)\pa_x+\omega\cdot\pa_y=\omega\cdot\pa_y$, cf.\ the
argument of $\tilde A^{l}$.

Now, if one regards two of $a_{11},a_{33},E^2$ as known, thus the
value of $l$ is
fixed to be the remaining single unknown, then,
corresponding to \eqref{eq:SU-mod-form}, at the artificial
boundary the standard principal symbol of $A_{l,\digamma}=e^{-\digamma/x}L\tilde Je^{\digamma/x}$ is, up to overall elliptic
factors, simply an integral with a weight given by the remaining
unknown's derivative, so for instance $\frac{\pa p_\pm}{\pa E^2}$ if $E^2$ is not known.
Namely, the principal symbol of $A_{l,\digamma}=e^{-\digamma/x}L\tilde Je^{\digamma/x}$ is an
integral over all $\omega$ (except in the special case when
$\zeta_3=0$, when only two values of $\omega$ enter) at covectors
$(z,\xi)$ where the covectors $\xi$ are the images of $(0,\omega)$
under the map
$H_z^{-1}$, which, as we recall, is the local inverse of $H_z$. Thus, if these partials are
either positive at all points or negative at all points, or simply
non-negative, resp.\ non-positive at all points with a strict sign at
{\em one} point, the principal
symbol is elliptic, since it is an integral of these expressions with
respect to a smooth positive measure, up to an overall elliptic
factor. 

Now, for $p_+$ all the partials
are non-zero as long as $\tilde\xi'$ and $\tilde\xi_3$ are
non-zero, with $\frac{\pa p_+}{\pa E^2}$ negative, the others positive; for $p_-$ the analogous claim holds for the $E^2$ partial,
and in addition for the $a_{11}$ and $a_{33}$ partials provided that
$E^2>0$.

For principal symbol computations we only need to consider the tangent
plane to the artificial boundary (and nearby level sets in the
interior); the question is whether the potential degeneracy of the
weights at $\tilde\xi'=0$ or $\tilde\xi_3=0$ provides an obstruction
to a strict sign at at least one point of relevance.

First consider the degeneracy of some of our
weights, such as $\frac{\pa
  p_+}{\pa a_{11}}$, at $\tilde\xi'=0$.

\begin{lemma}
If the gradient $\nabla
f$ of the transverse isotropy foliation function is not parallel to the artificial boundary,
points with $\tilde\xi'=0$ cannot give rise to vectors tangent to the
artificial boundary under Hamiltonian map $H_x$.
\end{lemma}

Note that the hypothesis on $\nabla f$ follows at least sufficiently
close to the actual boundary under the conditions we discussed for the
qSH waves, which we are assuming.

\begin{proof}
With an abuse of notation $\sum_j\frac{\pa
  p_\pm}{\pa\xi_j}\pa_{x_j}$ is $\sum_j\frac{\pa
  p_\pm}{\pa\tilde\xi_j}\pa_{\tilde x_j}$, both being the base
component, i.e. pushforward to the base, of the Hamilton vector
field, expressed in different coordinates.
But if $\tilde\xi'$ vanishes, this vector is a multiple of
$\pa_{\tilde x_3}$, so is parallel to the axis of transverse isotropic elasticity. Correspondingly, if the gradient $\nabla
f$ of the transverse isotropy foliation function is not parallel to the artificial boundary, which we are assuming, then $\tilde\xi'$ cannot
vanish at the relevant points as $\pa_{\tilde x_3}$ cannot be both parallel
to the axis {\em and} tangent to the level sets of the convex
foliation function.
\end{proof}

Thus, if we have a weight which is
non-negative and only
vanishes if $\tilde\xi'=0$, such as $\frac{\pa
  p_+}{\pa a_{11}}$, arising when we are attempting to
recover $a_{11}$ from $p_+$ travel times, we indeed have ellipticity at the standard principal
symbol level, i.e., $A_{l,\digamma}$ is elliptic in the standard sense
when $l$ corresponds
to $a_{11}$ and the wave speed used is $p_+$. {\em Together with Lemma~\ref{lemma:finite} below, this
  proves that the qP travel times determine $a_{11}$, in the sense of Theorem~\ref{thm:p-sv}.}

Unfortunately, there are points in the tangent plane at the artificial
boundary with $\tilde\xi_3=0$, however, which is an issue for the
determination of $E^2$ and $a_{33}$. Indeed, this is exactly the
statement that there are points in the tangent plane annihilated by
the differential
$df$ of the foliation function, i.e. orthogonal to $\nabla f$, which
happens even under the conditions we discussed for the qSH waves.
Since we still have a fixed though degenerate sign for our weight (so
it can vanish, but is $\geq 0$ everywhere or $\leq 0$ everywhere), this is
only an issue if the weight vanishes at every point at which the weight is evaluated in the
principal symbol.

Note the vanishing phenomenon for the weight occurs for all relevant
covectors even away from the
artificial boundary:

\begin{lemma}
Away from the artificial boundary, in $x>0$, the only points $\zeta$
for which
$\tilde\xi_3=0$ at all points on the critical set near the tangent
space of the foliation giving the artificial boundary are those in the
span of $df$.
\end{lemma}

\begin{proof}
First,
at covectors in the span of $df$, we are integrating
over integral curves with tangent vectors annihilated by $df$, but at
all of these $\tilde\xi_3=0$. On the other hand, for any covector $\zeta$ not
in the span of $df$, the set of `bad vectors' annihilated by both
$\zeta$ and $df$ is a
line, so in any open set of vectors annihilated by $\zeta$ (which thus
form a 2-dimensional family), such as those in an arbitrarily
small neighborhood of the tangent space to the level set of the
foliation, there will be vectors which are {\em not} in the `bad set',
proving the lemma.
\end{proof}

Thus, {\em away from the boundary ellipticity can only fail at points in
  the span of $df$}, where we have already seen that it {\em does
  fail}. That the failure is quadratic follows simply from the fixed
(degenerate) sign of the principal symbol.

We now turn to the non-degeneracy of the quadratic vanishing:

\begin{lemma}
Suppose that the hypothesis of Lemma~\ref{lemma:non-deg-diff-xi3} holds.  
For $A_{l,\digamma}$ corresponding to $E^2$ and $a_{33}$, the quadratic
vanishing of the principal symbol at the span of $df$ is non-degenerate in $x>0$.
\end{lemma}

\begin{proof}
The lemma follows from showing that along any line
transversal to the span of $df$ the quadratic vanishing is
non-degenerate, i.e., the second derivative is strictly positive (or
strictly negative) since then the a priori positive (or negative)
indefinite nature of Hessian combined with this fact implies positive
(or negative) definiteness.

But this can be seen as follows: consider $\nu$
not in the span of $df$ and $\zeta=\zeta_\ep=G_0(df)^{-1}\,df+\ep\nu$, where one may assume
that $\nu$ is $G_0$-orthogonal to $df$ and of unit length; the desired
non-degeneracy follows if we find a vector annihilated by $\zeta$ and
close to the tangent space of the convex foliation which is the image,
via the Hamilton vector field map $H_{\tilde x}$, of a covector $\tilde\xi=\tilde\xi_\ep$ that
has $|\tilde\xi_3|\geq C\ep$, $C>0$ (independent of $\ep$), for then the fact that the relevant
weights are non-degenerate multiples of $\tilde\xi_3^2$ proves the
conclusion.
Since the $H_{\tilde x}$ maps covectors with
vanishing third component to $\Ker df$, and $H_{\tilde x}$ has, by Lemma~\ref{lemma:non-deg-diff-xi3}, invertible differential
at $\tilde\xi_3=0$, we see that
for a vector $v'$ there is a  covector mapped to it by
$H_{\tilde x}$ whose
the third (tilded) component is a
non-degenerate multiple of the distance (with respect to any positive
definite inner product on the tangent space) of $v'$ from the kernel
of $df$. Hence, it
suffices to show that $\Ker\zeta$ contains a vector $v'$ which is
$\geq C\ep$, $C>0$, distance from $\Ker \,df$ but is still near the
tangent plane to the convex foliation (so that it is within the
support of the cutoff).

This final claim can be seen as follows: by linear independence, for
$\ep\neq 0$,
$\Ker df$ and $\Ker\zeta_\ep$ intersect in a line in an angle $\sim\ep$
(more precisely, the tangent of the angle is $\ep$), and in any compact
`annulus' (closed ball minus a smaller open ball) centered at a point
in $\Ker df$ at a fixed distance
from the intersection, the distance between a point in $\Ker\zeta_\ep$
and $\Ker df$ is bounded below by $C\ep$, $C>0$ (and above by a
similar expression), so consider a non-zero vector $v$ in the tangent space of the convex
foliation which is annihilated by $df$; 
the tangent
space of the foliation is 2-dimensional and $df$ is not conormal to the level sets of
the convex foliation, so the span of $v$ is well-defined (i.e., there
is no freedom of choice as far as the span of $v$ is concerned); in any fixed
ball around it there is then a vector $v_\ep$ in $\Ker\zeta_\ep$ which
is distance bounded below by $C\ep$ (and above by a similar
expression) from $\Ker df$, proving the claim.
\end{proof}

At the artificial boundary a bit more care is required, and it
requires an explicit discussion of scattering covectors and maps
related to them. Since elsewhere in the paper only the statement of
the lemma is used, we do not recall the background here in more
detail, but see for instance \cite{Uhlmann-Vasy:X-ray},
\cite{Stefanov-Uhlmann-Vasy:Boundary}. The argument presented below is
a modification, keeping track of potential degeneracies in
identifications, of the arguments discussed above for the case of
points away from the artificial boundary.

\begin{lemma}
Suppose that the hypothesis of Lemma~\ref{lemma:non-deg-diff-xi3} holds. 
For $A_{l,\digamma}$ corresponding to $E^2$ and $a_{33}$,
the quadratic
vanishing of the principal symbol is non-degenerate near the
artificial boundary as well.
\end{lemma}

\begin{proof}
Consider a
scattering covector $\zeta=\zeta_3\frac{dx}{x^2}+\zeta'\cdot\frac{dy}{x}$ which is not in the span of
$x_3^{-1}\,df$, $x_3$ the convex level set function defining the
boundary. If $\zeta_3\neq 0$, the integral giving the principal symbol
contains contributions corresponding to the whole tangent
space of the boundary, which cannot lie completely in the kernel of
$df$ since $df$ is not conormal to the artificial boundary, so there
are points at which the weight is evaluated but $\tilde\xi_3\neq 0$.
On the other hand, if $\zeta_3=0$, i.e., we are working with a
scattering cotangent vector which is scattering cotangent to the
boundary, then as already discussed, there is a line, given by the
kernel of $\zeta'$, within the tangent space of the boundary
within which the weights get evaluated; this line needs to be in the
kernel of $df$ to lose ellipticity. But the kernel of $df$ within the
tangent space to the boundary is also one
dimensional (since $df$ is not conormal to the boundary), and includes
the kernel of the (non-zero!) projection of $df$, so it is exactly the
latter. Thus, ellipticity fails exactly if these two are the same,
i.e. exactly if $\zeta$ is a multiple of the image of $x_3^{-1}\,df$ in the
scattering cotangent bundle.

Again, the quadratic nature of the
vanishing of the principal symbol follows from the fixed, though
degenerate, sign of the principal symbol.

Finally, the non-degeneracy
of the quadratic vanishing can be seen by an argument broadly similar
to the one given above away from the boundary. Although we have
scattering pseudodifferential operators to consider, so their standard
principal symbols are homogeneous functions on the scattering
cotangent bundle with the zero section removed, it is beneficial to
work on the b-cotangent bundle: the two are related by a conformal
rescaling by $x_3$, so the cosphere bundles are exactly the same, and
utilizing the b-cotangent bundle we spare ourselves from explicitly writing $x_3$ or
$x_3^{-1}$ in many places: the relevant identification is the map
$\pi: T^*M\to\Tb^*M$ which is the adjoint of the smooth linear bundle map $\iota:\Tb M\to
TM$ which at $\pa M$ regards a vector tangent to $\pa M$ as simply a vector in
$T_{\pa M} M$, and thus is neither injective (the kernel is the span
of $x_3\pa_{x_3}$) nor surjective at the boundary. (The scattering
analogues, used in the first paragraph of the proof,
are $x_3^{-1}\pi$ and $\iota x_3^{-1}$.) Now, at each point $q$ in $\pa M$, $\pi(df)$ is an element of
$T^*_q\pa M$ (a well-defined subspace of $\Tb^*_q M$, unlike the case
of $T^*_q M$, within which the conormal bundle is well-defined!) thus
annihilates $\Nb_q M$ (a well-defined subspace of $\Tb_q M$ which is
spanned by $x_3\pa_{x_3}$), so the image under $\iota$ of
$\Ker\pi(df)$ is a line in $T_q M$ contained in $T_q\pa M$; this is
the line spanned by any nonzero vector $v$ in $\Ker(df)\cap T_q\pa M$. We again consider a family
$\zeta_\ep=\pi(df)+\ep \nu$ where $\nu\in\Tb^*_qM$, and we may assume
that $\nu$ is orthogonal to $\pi(df)$ with respect to an inner product
(dual b-metric) on $\Tb^*_qM$. Then $\Ker\zeta_\ep$ and $\Ker \pi(df)$
again meet in a line in an angle $\sim\ep$ (for $\ep\neq 0$, and the
angle is with respect to the aforementioned b-metric, though the
$\sim\ep$ statement is independent of the choice of the b-metric), and
the localization via the cutoff means that we are working in a fixed small neighborhood
of $T^*_q\pa M$ in $\Tb^*_q M$, which in particular includes a
neighborhood of the aforementioned $v$, in which, completely similarly
to above, in any fixed annulus (with respect to the b-metric) there
are points $v_\ep$ in $\Ker\zeta_\ep$ which are distance $\sim\ep$ from $\Ker
\pi(df)$. Since $\Ker\pi(df)$ contains the kernel of $\iota$, the
image under $\iota$ of $v_\ep$ is still
$\sim\ep$ distance away from $\iota(\Ker(\pi(df)))$. But this then
finally implies that $\iota(v_\ep)$ is distance $\sim\ep$ away from
$\Ker df$ itself, since this is a plane intersecting $T_q\pa M$ the
line $\iota(\Ker(\pi(df)))$ in a fixed non-zero angle. This shows that
the $\tilde\xi_3$ component of the covector corresponding to
$\iota(v_\ep)$ is $\geq C\ep$, $C>0$, which proves the non-degeneracy
as in the case away from the boundary.
\end{proof}

We also need to have an elliptic boundary principal symbol at finite 
points:

\begin{lemma}\label{lemma:finite}
Suppose that the gradient $\nabla f$ of the anisotropic layer function
 $f=\tilde x_3$ is neither parallel nor orthogonal to the
artificial boundary.
The boundary principal symbol for determining any one of
$a_{11},a_{33},E^2$ from $p_+$, as well as for determining $E^2$ from $p_-$, is elliptic at finite points.  For determining one of
$a_{11},a_{33}$ from $p_-$ the corresponding statement holds if
$E^2>0$.
\end{lemma}

\begin{proof}
For this
we recall the computation from \cite{Uhlmann-Vasy:X-ray} in the form
used in \cite[Proof of
Lemma~3.5]{Stefanov-Uhlmann-Vasy:Boundary}. For this one again writes the
projected bicharacteristics in the form
$$
(x+\lambda t+\alpha
t^2+O(t^3),y+\omega t+O(t^2)),
$$
where
$\alpha=\alpha(x,y,\lambda,\omega)$. Further, one computes the
integral \eqref{eq:aj-1-form-pre-xi} at $x=0$ with a
Gaussian weight function in place of $\chi$ (which one eventually
approximates by a compactly supported $\chi$) with the parameter $\nu$, which we choose to be
$\nu=\digamma^{-1}\alpha$, $\digamma>0$ to be chosen sufficiently
large. This gives, e.g.\ for $E^2$,
$$
(\zeta_3^2+\digamma^2)^{-1/2}\int_{\sphere^1}\nu^{-1/2}e^{-(\hat
  Y\cdot\zeta')^2/(2\nu(\zeta_3^2+\digamma^2))}\,\frac{\pa p_\pm}{\pa E^2}(x,\xi)\,d\hat Y,
$$
where the covector $\xi$ is the image of the tangent vector
$(\lambda,\omega)=(0,\hat Y)$ under 
$H_x^{-1}$, the local inverse of $H_x$. Unlike for the case
of the standard principal symbol, for which a stationary phase
computation was used, here there is no critical set to restrict to,
i.e., we are integrating with $\frac{\pa p_\pm}{\pa E^2}$ evaluated at
the images of {\em all} tangent vectors to the artificial boundary. This expression is positive, resp.\ negative, if
$\frac{\pa p_\pm}{\pa E^2}(x,\xi)\geq 0$, resp.\  $\leq 0$, for all
relevant $\xi$, with the inequality definite for at least one of
them; the relevant $\xi$ are the images of $(0,\omega)$ under $H_x^{-1}$. But, taking into account
\eqref{eq:material-derivs}, this
is the case for both $p_+$ and $p_-$, with the definiteness coming
from $\tilde\xi'$ and $\tilde\xi_3$ both being non-zero at one such
image, since the non-parallel nature of $\nabla f$ to the
artificial boundary means that $\tilde\xi'$ indeed never vanishes on
the preimage, while the vanishing of $\tilde\xi_3$ at a point would
mean that the corresponding tangent vector is orthogonal to $\nabla
f$, which in turn cannot happen everywhere as $\nabla f$ is not orthogonal to the
artificial boundary. A completely analogous conclusion holds for
$\frac{\pa p_+}{\pa a_{11}}(x,\xi)$ and $\frac{\pa p_+}{\pa
  a_{33}}(x,\xi)$, and if in addition $E^2>0$, also for $\frac{\pa p_-}{\pa a_{11}}(x,\xi)$ and $\frac{\pa p_-}{\pa
  a_{33}}(x,\xi)$.
\end{proof}


{\em The conclusion is that, with the others taken as known,
the operator $e^{-\digamma/x}L\tilde Je^{\digamma/x}$ is elliptic at finite points for any one of
$E^2,a_{11},a_{33}$ for the qP-travel time data, and $E^2$ (as well as
$a_{11},a_{33}$ if $E^2>0$) from the qSV-travel time data,
while the standard principal symbol ellipticity holds for $a_{11}$
from the qP-travel time data. Hence, taking into account Proposition~\ref{prop:J-tilde-J}, $a_{11}$ can be recovered from
the qP-travel time data
under the hypothesis that the anisotropic layers are not aligned with
the convex foliation.} This proves Theorem~\ref{thm:p-sv}.

Corollary~\ref{thm:a11-fn} is a simple extension of this:

\begin{proof}[Proof of Corollary~\ref{thm:a11-fn}.] We suppose
that there is a functional relationship between $a_{11}, a_{33}, E^2$,
concretely, $a_{33} = F(a_{11})$ and $E^2 = H(a_{11})$, with $F, H$
smooth, and suppose that $F' \geq 0$. We claim that then the qP and qsV travel times
determine $a_{11}$ (and thus all the others).

To show this, we take the sum of the qP and qSV travel times. This
cancels the $\pm$ in the equations \eqref{eq:material-derivs}, as we
compute below. Namely, the effective coefficient in the
pseudolinearization for $a_{11}$ becomes
\begin{equation*}\begin{aligned}
&\frac{\pa p_+}{\pa a_{11}}+\frac{\pa p_+}{\pa a_{33}}
F'(a_{11})+\frac{\pa p_+}{\pa E^2}H'(a_{11})\\
&\qquad+\frac{\pa p_-}{\pa a_{11}}+\frac{\pa p_-}{\pa a_{33}}
F'(a_{11})+\frac{\pa p_-}{\pa E^2}H'(a_{11})\\
&=
2|\tilde\xi'|^2+2F'(a_{11})\tilde\xi_3^2\geq 2|\tilde\xi'|^2,
\end{aligned}\end{equation*}
so the above argument for recovering $a_{11}$ given the other
parameters works equally well. Indeed, if $F'>0$ then the right-hand
side can be replaced by $2|\tilde\xi|^2$, so in fact the above
argument for $a_{11}$ can be shortened somewhat.  Note that there is
no need for assuming anything about the derivative of $H$, only the
existence of such a functional relationship, since $H$ cancels from
the computation of the pseudolinearization coefficient at the
boundary. (But note that $H$ overall enters into the
pseudolinearization formula, so the {\em existence} of $H$ is
crucial.)
\end{proof}

\section{Determining more than one parameter at a time}

Of course, one would like to determine more than one of these
ideally. Since we have two linear transforms, given by $p_\pm$, and
since at some covectors the standard principal symbol behavior of
these transforms only involves evaluation of the material derivative
of $p_\pm$ at two points with identical behavior (antipodal), even
with further modifications, as mentioned above in analogy with the
tensor transform, {\em one cannot expect to recover all three in an
  elliptic manner}. However, it is reasonable to recover two of the
three (or all three if two determine the third in a suitable manner);
for this one needs a linear independence statement for the principal
symbols which now must be considered a 2 by 2 matrix, with the inputs
being the material parameters, the outputs the data for the different
wave types $p_+$ vs.\ $p_-$. (One will need slightly more to implement
this, again cf.\ the modifications of the transform mentioned above.)
For instance, $\frac{\pa p_\pm}{\pa E^2}(x,\xi)$ (with $+$ considered
the first row, $-$ the second row of a column vector) and either
$\frac{\pa p_\pm}{\pa a_{11}}(x,\xi)$ or $\frac{\pa p_\pm}{\pa
  a_{33}}(x,\xi)$ are certainly linearly independent as long as
$\tilde\xi_3\neq 0$ and $\tilde\xi'\neq 0$ since the two expressions
$\frac{\pa p_\pm}{\pa E^2}(x,\xi)$ are negatives of each other, which
is not the case for the other ones. In order to implement this, one
defines $L$ as
$$
Lv(z)=x^{-2}\int\chi(\lambda/x) \begin{pmatrix}C_{1+}(z,\lambda,\omega) &C_{1-}(z,\lambda,\omega)\\C_{2+}(z,\lambda,\omega) &C_{2-}(z,\lambda,\omega) \end{pmatrix}\begin{pmatrix}v_+(\gamma^+_{z,\lambda,\omega})\\v_-(\gamma^-_{z,\lambda,\omega})\end{pmatrix}\,d\lambda\,d\omega,
$$
where the first index of $C_{i\pm}$ refers to the parameter being
recovered (first vs.\ second) and $\pm$ to the type of wave being
used. Calling the parameters $\mu_1$ and $\mu_2$, we need that the
integral in $\omega$ of
$$
\begin{pmatrix}C_{1+}(z,0,\omega)
  &C_{1-}(z,0,\omega)\\C_{2+}(z,0,\omega)
  &C_{2-}(z,0,\omega) \end{pmatrix}
\begin{pmatrix}\frac{\pa p_+}{\pa \mu_1}(z,\xi)&\frac{\pa p_+}{\pa
    \mu_2}(z,\xi)\\ \frac{\pa p_-}{\pa \mu_1}(z,\xi)&\frac{\pa p_-}{\pa \mu_2}(z,\xi)\end{pmatrix}
$$
over the circle with a positive weight is elliptic; here $\xi=\xi(z,\omega)$ is
determined from $(0,\omega)$ as above. Now we can choose the $C$
matrix to be simply the transpose of the second, material sensitivity
matrix, at the actual boundary (where it is known!), and extend in a
smooth manner into the interior. Then the integrand is positive
definite over the boundary except where $\tilde\xi'=0$ or
$\tilde\xi_3=0$ (where it vanishes), thus has positive definite symmetric part
even nearby in the interior, thus the integral also has positive
definite symmetric part. This proves the ellipticity of the boundary
principal symbol at finite points, provided
$$
\begin{pmatrix}\frac{\pa p_+}{\pa \mu_1}(z,\xi)&\frac{\pa p_+}{\pa
    \mu_2}(z,\xi)\\ \frac{\pa p_-}{\pa \mu_1}(z,\xi)&\frac{\pa
    p_-}{\pa \mu_2}(z,\xi)\end{pmatrix}
$$
is full-rank, which holds, as discussed already, e.g.\ if $\mu_1=E^2$,
$\mu_2$ one of the other parameters. We reiterate that if e.g.\ one
assumes that $a_{33}$ is a function of $a_{11}$ and $E^2$ (rather than
a priori known), very similar arguments work; in this case, assuming
e.g.\ $a_{33}=\phi(a_{11})$, for the sake of an example, one simply
needs that
$$
\begin{pmatrix}\frac{\pa p_+}{\pa E^2}(z,\xi)&\frac{\pa p_+}{\pa
   a_{11}}(z,\xi)+\frac{\pa p_+}{\pa
   a_{33}}(z,\xi)\phi' \\ \frac{\pa p_-}{\pa E^2}(z,\xi)&\frac{\pa
    p_-}{\pa a_{11}}(z,\xi)+\frac{\pa
    p_-}{\pa a_{33}}(z,\xi)\phi'\end{pmatrix}
$$
is full rank, which is the case if $\phi'>0$.

\medskip\medskip

{\em In summary, the problem of determining $E^2$ and either one of
  $a_{11}$ and $a_{33}$ (or both if there is an a priori known
  relationship between them) from the qP and qSV data under the
  hypothesis that the anisotropic layers are not aligned with the
  convex foliation is always elliptic at finite points, and
  ellipticity fails only at scattering covectors aligned with the
  projection of the tilt axis to the boundary as well as at span of
  the tilt axis interior.}

\def\cprime{$'$} \def\cprime{$'$}

\end{document}